\documentclass[11pt]{amsart}

\title{Invariant embeddings and weighted permutations}

\author{M.~Mastnak}
\address{Department of Mathematics and Computing Science, Saint Mary's University, 923 Robie St, Halifax, Nova Scotia, Canada B3N 1Z9}
\email{mmastnak@cs.smu.ca}

\author{H.~Radjavi}
\address{Department of Pure Mathematics, University of Waterloo, Waterloo, Ontario,
Canada N2L 3G1}
\email{hradjavi@uwaterloo.ca}


\thanks{M.~Mastnak partially supported by NSERC Discovery Grant 371994-2019.}

\def\M{\mathcal{M}}
\def\A{\mathcal{A}}

\def\X{\mathcal{X}}
\def\Mn{\mathcal{M}_n}
\def\Mnc{\mathcal{M}_n(\mathbb{C})}
\def\B{\mathcal{B}}

\def\H{\mathcal{H}}
\def\N{\mathcal{N}}
\def\G{\mathcal{G}}
\def\V{\mathcal{V}}
\def\W{\mathcal{W}}
\def\span{\operatorname{span}}

\def\CC{\mathbb{C}}

\newtheorem{theorem}{Theorem}
\newtheorem{corollary}[theorem]{Corollary}

\newtheorem{lemma}[theorem]{Lemma}
\newtheorem{proposition}[theorem]{Proposition}

\newtheorem{remark}[theorem]{Remark}
\newtheorem{example}[theorem]{Example}

\newtheorem{question}[theorem]{Question}

\begin{document}

\maketitle

\begin{abstract}  
We prove that for any fixed unitary matrix $U$, any abelian self-adjoint algebra of matrices that is invariant under conjugation by $U$ can be embedded into a maximal abelian self-adjoint algebra that is still invariant under conjugation by $U$.  We use this result to analyse the structure of matrices $A$ for which $A^*A$ commutes with $AA^*$, and to characterize matrices that are unitarily equivalent to  weighted permutations.  
\end{abstract}

\section{Introduction}

This short paper resulted from our work on the following two topics of natural, independent, interest. They happen to be related as shown below.

(a) Let $\mathcal{A}$ be an abelian semisimple subalgebra of $\Mnc$ that is invariant under conjugation by a group $\mathcal{G}$ of a matrices, i.e., for every $G$ in $\mathcal{G}$,
$$
G^{-1}\mathcal{A} G := \{G^{-1} A G : A\in\G\} \subseteq \A.
$$
Can $\A$ be embedded into a maximal semisimple commutative subalegbra of $\Mnc$ that is still invariant under conjugation by $\mathcal{G}$?

In particular, if $\A$ is an abelian self-adjoint subalgebra of $\Mnc$ and $U$ is a unitary matrix such that $U^*\A U\subseteq \A$, does there exist a maximal abelian self-adjoint algebra (masa) $\widehat{\A}$ such that $\A\subseteq\widehat{\A}$ and $U^*\widehat{\A} U\subseteq \widehat{\A}$?

This question can obviously  be asked about
self-adjoint algebras of operators on infinite-dimensional Hilbert space, but this will be studied in \cite{MR2}. In the present note we give an affirmative answer in the finite-dimensional setting, not just in this particular case of 
$\mathcal{G}=\langle U\rangle$, but more generally, when $\G$ is
completely reducible and its non-abelian subquotients (as abstract groups)
contain noncentral normal abelian subgroups.

(b) Which matrices are unitarily equivalent to weighted permutations, i.e., matrices of the form $DV$, where $D$ is diagonal and $V$ is a permutation matrix?

It is easy to see that such matrices $A$ are half-normal, that is, $A^* A$ and 
$A A^*$ commute. (The terms ``semi-normal", ``skew-normal" and ``quasi-normal", which are easier on the ear, have already been claimed.  They generalize the notion of normality in infinite dimensions, but in finite dimensions they are equivalent to normality.)  So, the question becomes: which half-normal matrices are
unitarily equivalent to weighted permutations?

Using our answers to (a), we show that $A$ is unitarily equivalent to a 
weighted permutation if and only if it is a limit of half-normal operators
$A_n$ with pairwise distinct singular values.  The singular values of T are, by definition, the eigenvalues of the positive part $P$ in the polar decomposition $PU$ of $T$.  These are, equivalently, the eigenvalues of $(T^* T)^\frac{1}{2}$
or those of $(T T^*)^\frac{1}{2}$.  

It is interesting that not every half-normal operator is a limit of this special kind, i.e., half-normal operators with pairwise distinct singular values.

\section{Preliminaries}

Even though many of the results of this paper hold for more general fields, we restrict ourselves, for the sake of simplicity, to the field of complex numbers $\CC$.  We use notation $\M_n(\CC)$ to denote the algebra of $n\times n$ matrices with entries in $\CC$. The standard basis vectors of $\CC^n$ are denoted by $e_1,\ldots, e_n$.  

We say that a collection $\X\subseteq\M_n(\CC)$ of matrices is irreducible if it has no common nontrivial invariant subspaces.  We say that it is completely reducible if we can decompose $\CC^n=\V_1+\ldots+\V_r$ into a direct sum of $\X$-invariant subspaces on which $\X$ acts irreducibly.  (We are aware of the irony that irreducible collections are also completely reducible.)

For an invertible $G\in\M_n(\CC)$ and a subset $\X\subseteq\M_n(\CC)$ we say that $\X$ is $G$-invariant if for every $X\in\X$ we have that $G^{-1}XG\in\X$.  If $\G$ is a collection of invertible matrices, then we say that $\X$ is $\G$ invariant if it is invariant under conjugation by every $G\in\G$.

By a weighted permutation or a monomial matrix we mean a matrix that has at most one nonzero entry in every column and every row.  Note that every weighted permutation $A$ can be written as $A=DP$, where $D$ is a diagonal matrix and $P$ is a permutation matrix (i.e., for some permutation $\pi\colon\{1,\ldots, n\}\to \{1,\ldots, n\}$ we have that $Pe_j=e_{\pi(j)}, j=1,\ldots, n$); we refer to the diagonal entries of $D$ as weights.  If $A$ is invertible then such a decomposition is unique.

We say that a collection $\X\subseteq\M_n(\CC)$ of matrices is monomializable if there is a basis of the underlying vector spaces $\V=\CC^n$ in which all matrices in $\X$ are weighted permutations.  
For a group $\G\subseteq\M_n(\CC)$ this is equivalent to demanding that the corresponding representation is induced from a sum of one-dimensional representations of a subgroup.  For group representations usually the term``monomial" is used for what we refer to as ``monomializable".

\section{Invariant embeddings of commutative semi-simple algebras}

In this section we explore the following question: Let $\G\subseteq\mathcal{M}_n(\mathbb{C})$ be a fixed group of matrices and let $\mathcal{A}\subseteq \M_n(\mathcal{C})$ be a semi-simple commutative algebra that is invariant under conjugation by $\G$.  Can we embed $\A$ into a maximal commutative semi-simple algebra (i.e., an algebra simultaneously similar to the algebra of all diagonal matrices) that is still $\G$-invariant?  An equivalent way to rephrase the above question is as follows: Let $\G\subseteq\M_n(\CC)$ be a group and let $\mathcal{X}\subseteq\M_n(\CC)$ be a set of diagonal matrices that is invariant under conjugation by $\G$.  Is there  a basis of $\V=\CC^n$ in which $\G$ consists of weighted permutations and $\X$ is still diagonal? (Note that if some set is invariant under conjugation by $\G$, then so is the algebra generated by the set.)
Without imposing additional assumptions, the answer is obviously negative.  For example, if the group $\H=\langle\G,\X\rangle$ generated by $\G$ and invertible elements of $\X$ is not monomializable (i.e., not simultaneously similar to a subgroup of weighted permutations).  In the extreme: the set $\X$ of all scalar matrices is obviously invariant under conjugation by $\G=GL_n(\CC)$, but clearly cannot be enlarged to a larger $\G$-invariant set.

We briefly discuss properties that $\G$ may posses (as an abstract group):
\begin{enumerate}
\item Every irreducible representation of $\G$ is monomializable
\item Every non-abelian subquotient of $\G$ has a noncentral abelian normal subgroup
\item Every non-abelian subquotient of $\G$ has a cyclic normal subgroup 
\item Every non-abelian subquotient of $\G$ has a nontrivial centre
\end{enumerate}

It is fairly easy to see (and it is implicit in the proofs of Chapter 3 of \cite{Brall}) that $(4)\implies (3)\implies (2)$.  If $\G$ is a finite group, then the condition (3) is equivalent to $\G$ being supersolvable (in fact, the condition (3) is how R.~Baer originally defined supersolvable \cite{Ba}) and the condition (4) is equivalent to $\G$ being nilpotent.  In cases (3) and (4) we need to check the condition for quotients only (and the conculusion for the sub-quotients follows)\cite{Ba}.

The (abstract) groups satisfying condition (1) are usually referred to as M-groups and sometimes also as monomial groups (usually the notion is described by demanding that every irreducible representation of $\G$ is induced from one-dimensional representations of subgroups).
 It is well known that this condition is not hereditary.  If we assume that condition (1) is satisfied for all subgroups of $\G$ (i.e., all irreducible representations of all subgroups of $\G$ are monomializable), then this is equivalent to condition (2) \cite[Theorem 3.4]{Brall} (the just mentioned theorem only establishes that (2) implies (1); the implication that any hereditary family of groups satisfying (1) must also satisfy (2) is implicit in some of the discussions of Chapter 3 in \cite{Brall} and is also easy to see). For the sake of completeness and also, to make this result more accessible to mathematicians less familiar with group representation theory, we present a coordinate version of the proof below.
\begin{theorem}[{\cite[Thm 3.4]{Brall}}]\label{Mgroup} Let $\G$ be a group.
If every nonabelian sub-quotient of $\G$ contains a noncentral abelian normal subgroup, then every irreducible reprentation of every subgroup is monomializable.
\end{theorem} 
\begin{proof} 
Suppose, toward a contradiction, that there is a non-monomializable irreducible group $\G\subseteq\M_n(\mathbb{C})$ such that  every nonabelian subquotient of $\G$ contains a noncentral abelian normal subgroup.  Assume also that $n$ is the smallest integer for which such a group exists.
Clearly this means that $\G$ is not abelian. Let $\N$ be a noncentral abelian normal subgroup.  By Clifford's Theorem (see \cite[Section 2.2]{MR1} for a coordinate version that we use below) we can assume that $n=rs$ with $r>1$ and that, up to simultaneous similarity, we have the following:
\begin{enumerate}
\item $\G$ is $r\times r$ block-monomial with all blocks of size $s\times s$.
\item The set $\G_{i,j}$ of nonzero elements in each block $(i,j)$
of $\G$ is individually equal to a fixed irreducible group $\G_{1,1}\subseteq\M_{s}(\CC)$. 
\item $\N$ is block diagonal and each diagonal block of $\N$ is scalar.
\end{enumerate}
(Note that for this we are not invoking the full strength of Clifford's Theorem.)  Since $s<n$ we must then, due to minimality of $n$, have that $\G_{1,1}$ is monomializable (since $\G_{1,1}$ can be identified with a sub-quotient of $\G$ and therefore all its nonabelian subqotiens contain noncentral abelian normal subgroups).
But then $\G$ must be monomializable as well.
\end{proof}
A small generalization of the above yields the following.  We apologize for duplicating much of the proof.
\begin{theorem}\label{alginvariant}
Let $\G\subseteq\M_n(\CC)$ be a completely-reducible group of invertible matrices such that every irreducible subrepresentation of every subgroup is monomializable.  Then every $\G$-invariant commutative semisimple algebra $\A$ embeds into a $\G$-invariant maximal commutative semisimple algebra $\B$.
\end{theorem}
\begin{proof}
Let $\N$ be the set of invertible matrices in $\A$ and let $\H=\langle \G,\N\rangle$ be the group generated by $\G$ and $\N$.  Note that $\H$ is completely reducible, that $\N$ is a normal subgroup of $\H$, and that $\H/\N$ is isomorphic to $\G/(\G\cap\N)$.  With no loss of generality we now assume that $\H$ is irreducible.  If the action (by conjugation) of $\G$ on $\N$ is trivial (i.e., all elements of $\N$ commute with all elements of $\G$), then we are done.  By the coordinate version of Clifford's Theorem discussed in the proof of Theorem \ref{Mgroup}
we can assume that $n=rs$ with $r>1$ and that, up to simultaneous similarity, we have the following:
\begin{enumerate}
\item $\H$ is $r\times r$ block-monomial with all blocks of size $s\times s$.
\item The set $\H_{i,j}$ of nonzero elements in each block $(i,j)$
of $\H$ is individually equal to a fixed irreducible group $\H_{1,1}\subseteq\M_{s}(\CC)$. 
\item $\N$ is block diagonal and each diagonal block of $\N$ is scalar.
\end{enumerate}
Note that $\H_{1,1}\subseteq \N_{1,1} \G_{1,1}\subseteq\CC \G_{1,1}$ and hence $\G_{1,1}$ is irreducible.  
Note that $\G_{1,1}$ can be viewed as an irreducible subrepresentation of $\G_{\W}=\{G\in\G: G(\W)\subseteq \W\}$ where $\W=\span\{e_1,\ldots, e_s\}$ and is therefore monomializable.  Assume with no loss of generality that $\G_{1,1}$  consists of weighted permutations.  Then in this basis $\N$ is diagonal and the entire $\G$ consists of weighted permutations.  Now let $\B$ be the algebra of all diagonal matrices (in this basis) and note that $\A\subseteq \B$ and that $\B$ is $\G$-invariant.
\end{proof}

\begin{corollary}\label{uniinvariant}
Let $\G\subseteq\M_n(\CC)$ be a group of unitary matrices such that every irreducible subrepresentation of every subgroup is monomializable.  Then every $\G$-invariant abelian self-adjoint algebra $\A$ embeds into a $\G$-invariant maximal commutative semisimple algebra $\B$.
\end{corollary}
\qed

\begin{corollary}
Let $\G\subseteq\M_n(\CC)$ be a completely reducible group and let $\A\subseteq\M_n(\CC)$ be a $\G$-invariant commutative semi-simple algebra. If every non-abelian subquotient of $\G$ contains a non-central abelian normal subgroup, then $\A$ is contained in a $\G$-invariant maximal commutative semisimple algebra $\B$.  If in addition we assume that $\G$ is unitary and $\A$ is self adjoint, then we can also assume that $\B$ is a masa.
\end{corollary}\qed

\begin{corollary}\label{cor5} Let $U$ be a unitary matrix and let $\A$ be a $U$-invariant abelian self-adjoint algebra.  Then $\A$ is contained in a $U$-invariant masa $\B$.
\end{corollary}\qed

\section{Half-normal operators and weighted permutations}

In this section we investigate the structure of matrices that are unitarily equivalent to weighted permutations.  The concept of half-normality is an important ingredient as weighted permutations are clearly such.

\begin{remark} The analogous question with similarity in lieu of unitary equivalence is fairly trivial.  Indeed, it is easy to see that a an invertible matrix is similar to a weighted permutation if and only it is diagonalizable.  For a general matrix $A$ we then get that $A$ is similar to a weighted permutation if and only if it is similar to a block diagonal matrix 
$$\begin{pmatrix} D & 0\\ 0 & N\end{pmatrix}, 
$$
where $D$ is diagonal and $N$ is nilpotent, or, equivalently, if all Jordan blocks of $A$ corresponding to nonzero eigenvalues are of size one.
\end{remark}

The following lemma describes half-normality in terms of polar decompositions.
\begin{lemma} Let $A=PU$ where $P$ is positive and $U$ is unitary.  Then the following are equivalent.
\begin{enumerate}
\item $A$ is half-normal
\item $P$ commutes with $U^* P U$
\item $P$ commutes with $UPU^*$.
\end{enumerate}
\end{lemma}
\begin{proof}
The result immediately follows from the following (well-known) observations:
\begin{itemize}
\item Since $P$ is positive we have that $P$ is a polynomial in $P^2$ and hence $P$ commutes with $U^* P U$ if and only if $P^2$ commutes with $U^* P^2 U$.
\item $AA^*=P^2$, $A^* A = U^* P^2 U$.
\end{itemize}
\end{proof}

\begin{lemma}  Let $A$ be a half-normal operator with 
pairwise distinct singular values.   Then $A$ is unitarily 
equivalent to a weighted permutation.   Moreover, if $A=PU$ is a polar decomposition into the product of
positive $P$ and unitary $U$, then for every $k\in\mathbb{N}$ the matrix $P$ commutes with $(U^*)^k P U^k$.
\end{lemma}
\begin{proof}
Let $A=PU$ be a polar decomposition.  Up to unitary equivalence, we can assume that $P$ is diagonal.  Since the diagonal entries of $P$ are pairwise distinct, we have that $U$ is a unitary weighted permutation and hence $A$ is a weighted permutation. Also, for every $k\in\mathbb{N}$, $U^k$ is again a unitary weighted permutation and hence $(U^*)^k P U^k$ is also diagonal and thus commutes with  $P$.
\end{proof}

\begin{theorem}
Let $A\in\Mnc$ be a half-normal operator.  Then the following are equivalent.
\begin{enumerate}
\item $A$ is unitarily equivalent to a weighted permutation.
\item $A$ is a limit of half-normal operators $(A_m)_{m=1}^\infty$ that have pairwise distinct singular 
values.
\item $A$ has a polar decomposition\footnote{Note that we are not assuming that $A$ is invertible and hence a polar decomposition is not unique.} $A=PU$ into a product of positive $P$ and unitary $U$ such that
for every $k\in\mathbb{N}$, the operator $(U^*)^k P U^k$ commutes with $P$.
\end{enumerate}
\end{theorem}
\begin{proof}
(1)$\implies$(2): Immediate by writing $A$ as a weighted permutation and suitably perturbing the weights.

(2)$\implies$(1):  Assume that $A$ is the limit of a sequence of half-normal operators $(A_m)_m$ each having pairwise distinct singular values.  By the lemma above there exist unitary matrices $V_m$ such that $V_m^* A_m V_m = D_mS_m$, where $D_m$ is diagonal and $S_m$ is a permutation matrix.  By moving to subsequences if necessary, we assume that for every $m$, $S_m$ is a fixed permutation matrix $S$ and that the sequence $(V_m)_m$ converges to a unitary matrix $V$.
It then immediately follows that the sequence $(D_m)_m$ converges to the matrix $D=V^* A V S^*$, which must automatically be diagonal.  Hence $V^* A V = DS$ is a weighted permutation.


(1) $\implies$ (3): Assume that $A$ is a weighted permutation and write it as a product of positive diagonal matrix  $P$ and a unitary weighted permutation $U$.  It is then obvious that for every $n\in\mathbb{N}$, $(U^*)^n P U^n$ is also a diagonal matrix and hence commutes with $P$.

(3) $\implies$ (1): Assume that $A=PU$ where $P$ is positive and $U$ is unitary such that for all $k\in\mathbb{N}$, the matrix $(U^*)^k P U^k$ commutes with $P$.  Let $\mathcal{A}$ be the self adjoint algebra generated by $\{ (U^*)^k P U^k : k\in\mathbb{N}\}$.  By the above assumption $\mathcal{A}$ is abelian.  Since $\mathcal{A}$ is invariant under conjugation by $U$, it is contained in a masa $\mathcal{D}$ that is also invariant under conjugation by $U$ by Corollary \ref{cor5}.  It follows that $P$ is a limit of elements $P_m$ of $\mathcal{D}$ that have pairwise distinct eigenvalues.  Hence $A$ is the limit of $A_m=P_m U$.  Note that for every $m$, $A_m$ has pairwise distinct singular values.  Since (1)$\iff$ (2), this means that $A$ is unitarily equivalent to a weighted permutation.
\end{proof}

The next proposition shows that the condition (3) in the theorem above only needs to be checked for $k\le n-1$.  This is sharp as the example following the proposition shows.

\begin{proposition} Let $U\in\Mn(\CC)$ be unitary and let $P\in\Mn(\CC)$ be self-adjoint such that $P$ commutes with $(U^*)^k P U^k$ for all $k=1,\ldots, n-1$. Then $P$ commutes with $(U^*)^k P U^k$ for all $k\in\mathbb{N}$.
\end{proposition}
\begin{proof} 
If $P=I$, then the conclusion is vacuously satisfied.  Now assume that $P\not=I$.
For $k\ge 0$ let us abbreviate $P_k=(U^*)^k P U^k$.  
Let $\A$ be the unital algebra generated by the commuting set
$\{P_0,\ldots, P_{n-1}\}$.   Since $\A$ is commutative and self-adjoint it has dimension at most $n$.  Hence the set $I,P_0,\ldots, P_{n-1}$ is linearly dependent.  Let $\ell$ be the smallest integer such that the set $I,P,\ldots,P_{\ell}$ is linearly dependent and note that this means that $P_\ell$ is a linear combination of $I,P_0,\ldots, P_{\ell-1}$.  It is then easy to see that for every $k\ge\ell$ (and hence also for very $k\ge n-1$) we have that $P_k$ is a linear combination of $I,P_0,\ldots, P_{\ell-1}$ and therefore must commute with $P$.
\end{proof}

The following example shows that $n-1$ in the proposition above is best possible.
\begin{example} Let $C$ be the $n\times n$ cycle matrix, i.e., the matrix such that $Ce_1 = e_n$ and for $k=2,\ldots, n$ we have $Ce_k=e_{k-1}$.  Furthermore, let
$$V=\frac{1}{\sqrt{2}}\begin{pmatrix} 1 & -1 \\ 1 & 1\end{pmatrix}
\mbox{, let }
W=\begin{pmatrix} I_{n-2} & 0\\0 & V\end{pmatrix}
\mbox{, and let } U=WC.$$  
Let $P=E_{1,1}$ (where $E_{r,s}$ is 
the matrix unit with $1$ in the $(r,s)$-entry and zeroes everywhere else).  
Then for $k=1,\ldots, n-2$ we have that $(U^*)^k P U^k$ commutes with $P$, but $(U^*)^{n-1} P U^{n-1}$ does not commute with $P$.  Indeed, for $k=1,\ldots, n-2$ we can see inductively that $(U^*)^k P U^k=E_{k+1,k+1}$.  Let
$$
Q=V^* \begin{pmatrix} 1 & 0\\0 & 0\end{pmatrix} V = \frac{1}{2}\begin{pmatrix}
1 & -1 \\-1& 1\end{pmatrix}.$$
Now note that the matrix
$$
(U^*)^{n-1} P U^{n-1}=U^* E_{n-1,n-1} U=C^* \begin{pmatrix} 0 & 0 \\ 0 & Q\end{pmatrix} C$$ cannot commute with $P$ since the matrix
$$
C Q C^*= C^* \begin{pmatrix} 0 & 0 \\ 0 &  Q\end{pmatrix} C
$$
does not commute with $CPC^* = E_{n,n}$.
\end{example}

The following is a modification of the example above.  It shows an invertible  $4\times 4$ matrix $A$ with polar decomposition $A=PU$ such that $P$ commutes with $U^* P U$ (hence $A$ is half-normal), but $P$ does not commute with $(U^*)^2 P U^2$ (and hence $A$ cannot be unitarily equivalent to a weighted permutation).
\begin{example} Let $B=\frac{1}{\sqrt{2}}\left(\begin{array}{cc}
1 & 1 
\\
 -1 & 1 
\end{array}\right)$, $C=\left(\begin{array}{cc}
0 & 1 
\\
 1 & 0 
\end{array}\right)$ and let $U_1, U_2$ be the block diagonal $4\times 4$ matrices $U_1=\left(\begin{array}{cc}
B & 0 
\\
 0 & B 
\end{array}\right), U_2=\left(\begin{array}{ccc}
1 & 0 & 0 \\ 
0 & C & 0 \\
0 & 0 & 1 
\end{array}\right)$.  Now let
$$
P=\left(\begin{array}{cccc}
1 & 0 & 0 & 0 
\\
 0 & 1 & 0 & 0 
\\
 0 & 0 & 2 & 0 
\\
 0 & 0 & 0 & 2 
\end{array}\right)\mbox{ and } U=U_1U_2= \left(\begin{array}{cccc}
\frac{\sqrt{2}}{2} & 0 & \frac{\sqrt{2}}{2} & 0 
\\
 -\frac{\sqrt{2}}{2} & 0 & \frac{\sqrt{2}}{2} & 0 
\\
 0 & \frac{\sqrt{2}}{2} & 0 & \frac{\sqrt{2}}{2} 
\\
 0 & -\frac{\sqrt{2}}{2} & 0 & \frac{\sqrt{2}}{2} 
\end{array}\right).
$$
Then $A=PU$ is half-normal, but is not unitarily equivalent to a weighted permutation.  Indeed, a direct computation shows that $P$ and $UPU^*$ commute, but $P$ and $U^2 P (U^*)^2$ do not as 
$$(U^2 P (U^*)^2) P - P (U^2 P (U^*)^2) = \left(\begin{array}{cccc}
0 & 0 & \frac{1}{2} & 0 
\\
 0 & 0 & 0 & \frac{1}{2} 
\\
 -\frac{1}{2} & 0 & 0 & 0 
\\
 0 & -\frac{1}{2} & 0 & 0 
\end{array}\right).$$
\end{example}

\begin{question} If $A$ is unitarily equivalent to a weighted permutation, then it is clear that each power of $A$ is also equivalent to a weighted permutation and is hence half-normal.  Does the converse hold, i.e., if every power of $A$ is half-normal, must $A$ be unitarily equivalent to a weighted permutation?  What if we additionally assume that $A$ is invertible?
\end{question}

\end{document}